\documentclass[12pt]{article}
\usepackage{latexsym, amsmath, amsfonts, amsthm, amssymb}
\usepackage{times}
\usepackage{a4wide}
\usepackage{times}

\def \det {\text{\rm det}}
\def \tr {\text{\rm Tr}}
\def \Tr {\text{\rm Tr}}

\def \sgn {\text{sgn}}

\def \RR {\mathbb R}

\def \ZZ {\mathbb Z}

\def \eps {\varepsilon}
\def \vphi {\varphi}
\def \div {{\textrm{div}}}

\newtheorem{theorem}{Theorem}[section]

\newtheorem{lemma}[theorem]{Lemma}
\newtheorem{proposition}[theorem]{Proposition}

\newtheorem{corollary}[theorem]{Corollary}
\newtheorem{remark}[theorem]{Remark}

\def\myffrac#1#2 in #3{\raise 2.6pt\hbox{$#3 #1$}\mkern-1.5mu\raise 0.8pt\hbox{$
#3/$}\mkern-1.1mu\lower 1.5pt\hbox{$#3 #2$}}

\def\qed{\hfill $\vcenter{\hrule height .3mm
\hbox {\vrule width .3mm height 2.1mm \kern 2mm \vrule width .3mm
height 2.1mm} \hrule height .3mm}$ \bigskip}

\begin{document}
	
\title{Unimodal value distribution of Laplace eigenfunctions and a monotonicity formula}

\author{Bo'az Klartag}
\date{}

\maketitle

\begin{abstract}
Let $M$ be a compact, connected Riemannian manifold whose  Riemannian volume measure is denoted by $\sigma$. Let $f: M \rightarrow \RR$ be a non-constant eigenfunction of the Laplacian. The random wave conjecture suggests that in certain situations, the value distribution of $f$ under $\sigma$ is approximately Gaussian. Write $\mu$ for the measure whose density with respect to $\sigma$ is $|\nabla f|^2$. We observe that the value distribution of  $f$ under $\mu$ admits a unimodal density attaining its maximum at the origin.
Thus, in a sense, the zero set of an eigenfunction is the largest of all level sets.
When $M$ is a manifold with boundary, the same holds for Laplace eigenfunctions satisfying either the Dirichlet or the Neumann boundary conditions. Additionally, we prove a monotonicity formula for level sets of solid spherical harmonics, essentially by viewing nodal sets of harmonic functions  as weighted minimal hypersurfaces.
\end{abstract}
\section{Introduction}

Consider an eigenfunction $f$ of the Laplacian on a compact, $n$-dimensional, $C^{\infty}$-smooth Riemannian manifold $M$.
Let $\mu$ be a non-zero finite, Borel measure on $M$, and let $X$ be a random point in $M$, distributed according to the probability measure that is proportional to $\mu$.
The density function of the real-valued random variable $f(X)$ is referred to as the {\it value distribution density} of $f$ under $\mu$.

\medskip
One of the first measures to look at is of course the uniform measure $\sigma$, the Riemannian volume measure.
The random wave conjecture
of Berry \cite{berry} suggests
that under ergodicity assumptions, in the generic case ``a random wave is a random function''. That is, a Laplace eigenfunction corresponding to a large eigenvalue should have a value distribution density under $\sigma$ that is approximately Gaussian. We refer to  Jakobson, Nadirashvili and Toth \cite{JNT} and to Zelditch \cite{Z} for background
on eigenfunctions of the Laplacian on Riemannian manifolds.

\medskip Let us discuss the example of the standard flat torus $M = \RR^n / \ZZ^n$, even though its geodesic flow is not ergodic. In this case, the eigenfunctions of the simplest form are perhaps
\begin{equation}  f_k(x_1,\ldots,x_n) = \sin(2 \pi k x_1) \qquad \qquad (k=0,1,2\ldots), \label{eq_1032} \end{equation}
defined for $x = (x_1,\ldots,x_n) \in \RR^n / \ZZ^n$. Then $\Delta f_k = - 4 \pi^2 k^2 f_k$, and the eigenvalue  $-4 \pi^2 k^2$ tends to $-\infty$ with $k$. Nevertheless, the value distribution density of $f_k$ under the Riemannian volume measure is independent of $k$, and it equals
\begin{equation} t \mapsto \frac{1}{\pi \sqrt{1-t^2}} \qquad \qquad \qquad \text{for} \ t \in (-1,1). \label{eq_1020_} \end{equation}
This density function is unbounded, and its graph bears little resemblance to the graph of the Gaussian density $e^{-t^2/2} / \sqrt{2 \pi}$.
In fact, the density function in (\ref{eq_1020_}) is not even unimodal with a maximum at the origin, but quite the opposite.
Here  we note that this may be fixed if one modifies the Riemannian volume measure $\sigma$ in a simple manner.

\begin{theorem} Let $M$ be a compact, connected, Riemannian manifold, and let $f: M \rightarrow \RR$ be a non-constant Laplace eigenfunction.
Consider the measure $\mu$ whose density with respect to the Riemannian volume measure is the function $|\nabla f|^2$. Then the value distribution density of the function $f$ under $\mu$ is continuous and unimodal with a maximum at zero.

\medskip That is, the value distribution density $\psi: \RR \rightarrow [0, \infty)$ is a continuous function, supported on the interval $(\inf f, \sup f)$, strictly-increasing on $[\inf f, 0]$ and strictly-decreasing on $[0, \sup f]$.
\label{thm_1033}
\end{theorem}

An attempt at a physical interpretation of Theorem \ref{thm_1033}  goes roughly as follows: If $M$ is a  membrane vibrating according to the wave $f$ at a certain frequency, and
the area of the membrane is weighted by the local  energy $|\nabla f|^2$, then the contribution of a single level set  to the
total  energy of the wave  decreases as the level set gets further from the nodal set.

\medskip In the example described in  (\ref{eq_1032}), the density $\psi$ from Theorem \ref{thm_1033}
equals $(2/\pi) \cdot \sqrt{1-t^2}$ for $t \in [-1,1]$, which is indeed a continuous, unimodal function with a maximum at the origin.
Once stated, Theorem \ref{thm_1033} is almost trivial to prove, see Section \ref{sec3} below. Arguments similar to the proof of this theorem appear for example in Sogge and Zelditch \cite{SZ}. We were led to the formulation of Theorem \ref{thm_1033} by viewing zero sets of eigenfunctions as {\it weighted minimal surfaces} in the sense that we are about to describe.

\medskip Let $f$ be a function satisfying $\Delta f = -\lambda f$ in the manifold $M$. What optimization  problem is solved by its nodal set $Z = \{ x \in M ; f(x) = 0 \}$? The nodal set $Z$ is a smooth, $(n-1)$-dimensional submanifold outside a singular set of lower dimension, see Cheng \cite{cheng}.
We denote $\rho = |\nabla f|$ and treat $\rho$ as a weight function on the manifold $M$. We observe that the nodal  hypersurface  $Z$ is a critical point of the  weighted area functional
\begin{equation} N \mapsto \int_N \rho, \label{eq_1041} \end{equation}
where $N \subseteq M$ is an $(n-1)$-dimensional hypersurface in $M$ and the integral in (\ref{eq_1041}) is carried out with respect to the $(n-1)$-dimensional Hausdorff measure.
This is explained in detail in Section \ref{sec2}.
Critical points of the area functional in $\RR^3$ are minimal surfaces. Hence,
we say that the nodal set of any non-constant eigenfunction is a {\it weighted minimal hypersurface}. This suggests a connection between the study of nodal sets of Laplace eigenfunctions, and the theory of minimal surfaces and isoperimetry.

\medskip Furthermore,
let $t \in (\inf f, \sup f)$ be a value of the function $f$. Then the sublevel set $\{ f \leq t \}$ is a critical point of the functional
\begin{equation} \Omega \mapsto \left( \int_{\partial \Omega} \rho \right) - \lambda t \sigma(\Omega), \label{eq_146} \end{equation}
defined for domains $\Omega$ with a smooth boundary in $M$.  In other words, up to regularity issues each level set of $f$ has a  constant weighted mean curvature (``weighted CMC'') with respect to the weight $\rho = |\nabla f|$, see Section \ref{sec2} for definitions.
 The term  $\lambda t \sigma(\Omega)$ in (\ref{eq_146}) is a Lagrange multiplier. It corresponds  to the isoperimetric optimization problem, where one minimizes $\int_{\partial \Omega} \rho$ among all domains $\Omega$ of a fixed Riemannian volume measure $\sigma(\Omega) = \sigma (\{ f \leq t \})$.
It would be interesting to find geometric conditions guaranteeing that the sublevel sets of $f$ are not only critical points of the functional in (\ref{eq_146}), but are in fact global minimizers, the optimal solutions of this isoperimetric problem.

\medskip Recall that the isoperimetric profile of a compact, $n$-dimensional Riemannian manifold $M$ is
\begin{equation}  I(t) = \inf \{ Vol_{n-1}(\partial \Omega) \, ; \, Vol_n(\Omega) = t \} \qquad \qquad \qquad (0 < t < Vol_n(M)), \label{eq_826} \end{equation}
where $Vol_{n-1}$ is $(n-1)$-dimensional volume.
The minimizer $\Omega$ of the optimization problem in (\ref{eq_826}) has a boundary $\partial \Omega$ of constant mean curvature, up to regularity issues. Moreover, under curvature assumptions such as non-negativity of the Ricci tensor, the function $I$ is known to be concave and in particular unimodal. This was shown by Bavard and Pansu \cite{BP}, see the appendix of Milman \cite{milman} for a survey of related results.
The value distribution density $\psi$ of a Laplace eigenfunction with respect to $\mu$ is analogous to the isoperimetric profile, and according to Theorem \ref{thm_1033} the function $\psi$ is unimodal. By continuing the analogy, one is tempted to conjecture  that $\psi$ has certain concavity properties under some geometric assumptions on the Riemannian manifold $M$.

\medskip In Section \ref{sec4} we observe that Theorem \ref{thm_1033} admits a straightforward  generalization to the case of weighted manifolds, see Theorem \ref{thm_540} below. We move on to disucss manifolds with boundary and eigenfunctions satisfying either the Dirichlet boundary conditions or the Neumann boundary conditions.

\begin{theorem} Let $M$ be a connected, $n$-dimensional, Riemannian manifold with a
smooth $(n-1)$-dimensional boundary $\partial M$. Assume that $M \cup \partial M$ is compact.
Let $f: M \rightarrow \RR$ be a non-constant eigenfunction of the Laplacian, smooth up to the boundary.
	Consider the measure $\mu$ whose density with respect to the Riemannian volume measure is the function $|\nabla f|^2$. Assume one of the following:
	\begin{enumerate}
		\item[(i)] The case of Dirichlet boundary conditions:  The function $f$ vanishes on $\partial M$.
		\item[(ii)] The case of Neumann boundary conditions: Write $N$ for the outer unit normal of $\partial M$. Then $\langle \nabla f, N \rangle = 0$ on the boundary.
	\end{enumerate}	
	Then in each of these two cases,
	the value distribution density of the function $f$ under $\mu$ is strictly-increasing in $[\inf f, 0]$ and strictly-decreasing in $[0, \sup f]$.  \label{thm_952}
\end{theorem}

As an example, let us apply Theorem \ref{thm_952} in the case of the eigenfunction corresponding to the first eigenvalue of the   Laplacian with Dirichlet boundary conditions. Up to normalization, this eigenfunction is positive in $M$, and thus its value distribution density is supported on $[0, \sup f]$ and it is strictly-decreasing there. Of course, in this case the value distribution density is discontinuous at $0$.

\medskip A well-known result in the theory of minimal surfaces is the monotonicity formula (see e.g. Simon \cite[Chapter 4]{simon}.
In the particular case of complex submanifolds, the monotonicity formula goes back to
Lelong \cite{Lelong} and Rutishauser \cite{Ruti}). The following theorem provides a certain  monotonicity formula  for weighted minimal hypersurfaces such as level sets of $k$-homogeneous, harmonic functions in $\RR^n$. Write $B(x,r) = \{ y \in \RR^n \, ; \, |x-y| < r \}$ for $x \in \RR^n$ and $r > 0$.

\begin{theorem} Let $P: \RR^n \rightarrow \RR$ be a solid spherical harmonic of degree $k \geq 1$, that is, a $k$-homogeneous polynomial in $\RR^n$ with $\Delta P = 0$. Fix $t \in \RR$ and consider the
level set 	$M = \{ x \in \RR^n \, ; \, P(x) = t \}$. Then the function
\begin{equation}  r \mapsto \frac{1}{r^{n+k-2}} \int_{M \cap B(0, r)} |\nabla P| \qquad \qquad (r > 0) \label{eq_601} \end{equation}
is non-decreasing, where the integral is carried out with respect to the $(n-1)$-dimensional Hausdorff measure.
\label{thm_602}
\end{theorem}

In the case where $t = 0$ in Theorem \ref{thm_602}, by homogeneity the function  in (\ref{eq_601}) is constant in $r$. This constant is finite and non-zero.
This explains the choice of the exponent $n+k-2$ in (\ref{eq_601}).
Theorem \ref{thm_602} is proven in Section \ref{sec5}. See Corollary \ref{cor_929} below for a reformulation of Theorem \ref{thm_602} in terms of spherical harmonics on the unit sphere $S^{n-1} = \{ x \in \RR^n \, ; \, |x| = 1 \}$.
Theorem \ref{thm_602} implies the following corollary:

\begin{corollary} Let $P: \RR^n \rightarrow \RR$ be a solid spherical harmonic of a non-zero degree. For $t \in \RR$ set $Z_t = \{ x \in \RR^n \, ; \, P(x) = t \}$. Then the function
	\begin{equation}  t \mapsto \int_{Z_t \cap B(0,1)} |\nabla P| \qquad \qquad \qquad (t \in \RR) \label{eq_852} \end{equation}
	is unimodal: It is non-decreasing in $(-\infty,0]$ and non-increasing in $[0, \infty)$.
	
	\medskip Equivalently, the value distribution density of $P$ under the measure $\mu$ with density $|\nabla P|^2$ in the ball $B(0,1)$, 	
	is unimodal with a maximum at the origin. \label{cor_834}
\end{corollary}

Corollary \ref{cor_834} shows that we may replace the boundary conditions from Theorem \ref{thm_952} by homogeneity assumptions on the Laplace eigenfunction.
The monotonicity formula of Theorem \ref{thm_602} reminds us of the recent solution  of the Nadirashvili conjecture by Logunov \cite{lugnov}. According to \cite[Theorem 1.2]{lugnov}, for any harmonic function $f: \RR^n \rightarrow \RR$ and for any point $x_0$ with $f(x_0) = 0$,
\begin{equation}  Vol_{n-1}(M \cap B(x_0,r)) \geq c_n r^{n-1} \qquad \qquad (r > 0), \label{eq_256} \end{equation}
where $c_n > 0$ depends solely on the dimension $n$. In view of (\ref{eq_256}) and Theorem \ref{thm_602}, we may wonder whether a certain exact monotonicity formula holds true also for nodal sets of general harmonic functions in $\RR^n$. Perhaps numerical simulations can shed light here.

\medskip Throughout this note, we write
$\log$ for the natural logarithm,
$\langle \cdot, \cdot \rangle$ for the Riemannian scalar product in the tangent space $T_x M$, and $\nabla$ is the covariant derivative (Levi-Civita). We write $\Delta f$ for the Laplacian of $f$, and $\nabla^2 f$ is the Hessian of $f$ viewed as a $2$-covariant tensor, i.e., a bilinear form on the tangent space. We denote by $| \cdot |$  the usual Euclidean norm in $\RR^n$, and also $|v| = \sqrt{\langle v, v \rangle}$ for a tangent vector $v$ in a Riemannian manifold.
A hypersurface is a submanifold of codimension one.
A smooth function or hypersurface are $C^{\infty}$-smooth. Unless stated otherwise, by a manifold we mean a $C^{\infty}$-smooth manifold.
 We abbreviate $\{ f  \leq t \} := \{ x \, ; \, f(x) \leq t \}$ and $f(A) = \{ f(x) \, ; \, x \in A \}$.

\medskip {\it Acknowledgement}. I would like to thank Emanuel Milman for proposing an idea that has led to   simplification of the proof of Theorem \ref{thm_602} and Corollary \ref{cor_834}. I am grateful to  Sasha Logunov for  vivid explanations  on harmonic analysis, and  to David Jerison, Misha Sodin and Steve Zelditch for their interest
and for their remarks on an earlier version of this text.

\section{Nodal sets as weighted minimal hypersurfaces}
\label{sec2}

Let $(M, g, \rho)$ be an $n$-dimensional, weighted, Riemannian manifold. This means that $g$ is a Riemannian metric on the smooth manifold $M$, and that $\rho: M \rightarrow \RR$ is a smooth function
referred to as {\it weight}. Given a smooth hypersurface $Z \subseteq M$ and a unit normal $N$ to the hypersurface $Z$, we set
$$ H_{\rho} = \rho H - \langle \nabla \rho, N \rangle $$
where $H(x)$ is the usual mean curvature of the hypersurface $Z$ at the point $x$ with respect to the normal $N$,
the trace of the second fundamental form, i.e.,
$$ H(x) = -\Tr \left[ T_x Z \ni u \mapsto \nabla_u N \right]. $$
The quantity $H_{\rho}(x)$ is referred to as the {\it weighted mean curvature} of $Z$ at the point $x$ with respect to the normal $N$.
Note that replacing the normal $N$ by $-N$ reverses the sign of the weighted mean curvature. See Gromov \cite[Section 9.4.E]{Gr}, Morgan \cite{morgan} and references therein for background on the weighted mean curvature, in various normalizations. The advantage of our normalization of $H_{\rho}$ is the linear dependence on $\rho$, which allows the consideration of weights that are
not necessarily positive.

\medskip As is discussed in \cite{Gr, morgan}, the weighted mean curvature
is related to the first variation of weighted area. If the weighted mean curvature $H_{\rho}$ vanishes on a smooth hypersurface $Z \subseteq M$, then $Z$ is a critical point for the weighted area functional
$$ N \mapsto \int_N \rho $$
defined for smooth hypersurfaces $N \subseteq M$. This is a particular case of the next lemma.


\begin{lemma}
	Let $(M, g)$ be a  Riemannian manifold, let $U \subseteq M$ be an open set, and let $\rho: M \rightarrow \RR$ be
a smooth function, compactly-supported in $U$. Let $Z \subseteq M$ be such that $Z \cap U$ is a smooth hypersurface.
Let $V$ be a smooth vector field, compactly-supported in $U$, consider the flow $(\vphi_t)_{t \in \RR}$ that is induced by the vector field $V$,
and set $Z_t = \vphi_t(Z)$ for $t \in \RR$. Then,
	$$  \left. \frac{d}{dt} \int_{Z_t} \rho \right|_{t=0} = -\int_{Z \cap U} \langle N, V \rangle \cdot H_{\rho}, $$
	where all integrals are carried out with respect to the $(n-1)$-dimensional Hausdorff measure in $M$, and where $N$ is a unit normal to $Z$ with respect to which the weighted mean curvature $H_{\rho}$ is computed. \label{lem_1001_} 	
\end{lemma}

\begin{proof} Write $K$ for the closure of $\{ \rho \neq 0 \}$, a compact  contained in the open set $U$. The flow $(\vphi_t)_{t \in \RR}$ induced by the vector field $V$ satisfies
	\begin{equation} \frac{d \vphi_t(x)}{d t} = V(\vphi_t(x)) \label{eq_1023} \end{equation}
	and $\vphi_0(x) = x$. Thus $(\vphi_t: M \rightarrow M)_{t \in \RR}$ is a one-parameter group of smooth diffeomoprhisms.
Recall that $Z_t = \vphi_t(Z)$. By compactness, there exists $\delta > 0$ such that for $t \in (-\delta, \delta)$,
	\begin{equation} Z_t \cap K \subseteq \vphi_t(Z \cap U). \label{eq_415} \end{equation}
From now on, for $t \in (-\delta, \delta)$ we view $\vphi_t$ as a map whose domain is {\it only} the smooth hypersurface $Z \cap U$.
Hence $\vphi_t^* g$ is a Riemannian metric on $Z \cap U$ which is the pull-back of $g$ under $\vphi_t$.
	It follows from  (\ref{eq_415}) and from the change-of-variables $y = \vphi_t(x)$ that for any $t \in (-\delta, \delta)$,
	\begin{equation}
	\int_{Z_{t} } \rho
=
\int_{Z_{ t} \cap K} \rho
	= \int_{Z \cap U}
	 \rho(\vphi_t(x)) \sqrt{\det(\vphi_t^* g)^{\sharp}}
	\label{eq_423} \end{equation}
	where $(\vphi_t^* g)^{\sharp}  = A_t: T_x Z \rightarrow T_x Z$ is  defined via  $(\vphi_t^*g)(\alpha ,\beta) = \langle A_t \alpha, \beta \rangle$ for $\alpha, \beta \in T_x Z$.
	Differentiating (\ref{eq_423}) under the integral sign and using (\ref{eq_1023}) and the formula $d \log \det(A_t) / dt|_{t=0} = d \tr(A_t)/dt|_{t=0}$ we obtain
	\begin{equation}
	\left. \frac{d}{dt} \int_{Z_{ t}} \rho \right|_{t=0} = \int_{Z \cap U}  \left[ \left \langle \nabla \rho, V \right \rangle + \frac{\rho}{2}  \frac{d}{d t} \Tr( \vphi_t^* g)^{\sharp}|_{t=0} \right]. \label{eq_1151}
	\end{equation}
	By the definition of the Lie derivative, for any $\alpha,\beta \in T_x Z$,
	\begin{equation} \left. \frac{d}{d t}  \vphi_t^* g \right|_{t=0}(\alpha,\beta) = L_V g (\alpha, \beta)  = \langle \nabla_\alpha V, \beta \rangle + \langle \alpha, \nabla_{\beta} V \rangle, \label{eq_515} \end{equation}
	where the last passage is explained, e.g., in \cite[Section 2.62]{GHL}.
Write $\pi: T_x M \rightarrow T_x Z$ for the orthogonal projection operator, and let us decompose $V = V_N + V_T$ where $V_N = \langle V, N \rangle N$ and $V_T = \pi(V)$. Then by (\ref{eq_515}),
\begin{align*}   \frac{d}{d t} & \Tr( \vphi_t^* g)^{\sharp}|_{t=0}  = 2 \Tr \left[ T_x Z \ni u \mapsto \pi \nabla_u V \right]
\\ & = 2 \Tr \left[ T_x Z \ni u \mapsto \pi \nabla_u V_T \right] + 2 \langle V, N \rangle \Tr \left[ T_x Z_{t_0} \ni u \mapsto \nabla_u N \right]  = 2 {\rm div}_Z (V_T) - 2 \langle V, N \rangle H,
\end{align*}
where ${\rm div}_Z$ is the divergence operator on the smooth hypersurface $Z \cap U$. 
	From (\ref{eq_1151}) and the previous equation,
	$$
	\left. \frac{d}{dt} \int_{Z_{ t}} \rho \right|_{t=0} = \int_{Z\cap U}  \left[ {\rm div}_Z (\rho V_T) + \langle V, N \rangle \cdot \langle \nabla \rho, N \rangle -  
\langle V, N \rangle H \right] = -\int_{Z \cap U} \langle V, N \rangle \cdot H_{\rho}, $$
where the integral of ${\rm div}_Z (\rho V_T)$ vanishes by the divergence theorem, since $\rho V_T$ is a compactly-supported vector field in $Z \cap U$. 
This completes the proof.
\end{proof}

\begin{proposition}
	Let $(M, g)$ be a  Riemannian manifold and let $\rho, f: M \rightarrow \RR$ be smooth functions
	such that the closure of $\{ \rho \neq 0 \}$ is a compact contained in the open set $U = \{ |\nabla f| \neq 0 \}$. Abbreviate  $Z_t = \{ f =  t \}$.
	Then at any $t \in \RR$,
	\begin{equation}  \frac{d}{dt} \int_{Z_t} \rho  = -\int_{Z_{t} \cap U}  \frac{H_{\rho}}{|\nabla f|}, \label{eq_701} \end{equation}
	where all integrals are carried out with respect to the $(n-1)$-dimensional Hausdorff measure in $M$, and  the weighted mean curvature $H_{\rho}$ is computed with respect to the normal $\nabla f / |\nabla f|$. \label{lem_1001} 	
\end{proposition}

\begin{proof}
 Write $K$ for the closure of $\{ \rho \neq 0 \}$, a compact  contained in the open set $U$.
Since $\rho$ is supported in the compact $K \subseteq U$, and $Z_{t} \cap U$ is either empty or a smooth hypersurface, then necessarily  $\int_{Z_{t}} |\rho| < \infty$ for all $t \in \RR$.
Fix $t_0 \in \RR$ and let us prove (\ref{eq_701}) for $t = t_0$. If $Z_{t_0} \cap U = \emptyset$, then (\ref{eq_701}) holds trivially as $\int_{Z_t} \rho$ vanishes for $t$ near $t_0$. From now on in this proof, assume that  $$ Z_{t_0} \cap U \neq \emptyset. $$
Thus $Z_{t_0} \cap U$ is a smooth hypersurface.   
 Fix a smooth vector field $V$ in $M$ that is compactly supported in  $U$ and that satisfies $V(x) = \nabla f / |\nabla f|^2$ in a neighborhood of $K$.
	Consider the one-parameter group of diffeomorphisms $(\vphi_s)_{s \in \RR}$ induced by the vector field $V$. For any $x$ in a neighborhood of $K$ and for any $s \in \RR$ that is sufficiently close to zero,
	$$ f(\vphi_s(x)) = f(x) + \int_0^s \frac{d}{dr} f(\vphi_r(x)) dr = \int_0^s \langle \nabla f, V \rangle  dr = f(x) + s, $$
	since $\langle \nabla f, V \rangle = 1$ in a neighborhood of $K$. 	Hence there exists $\delta > 0$ such that for $t \in (t_0 - \delta, t_0 + \delta)$,
$$ Z_{t + t_0} \cap K = \vphi_t(Z_{t_0} \cap U) \cap K = \vphi_t(Z_{t_0}) \cap K . $$
Therefore,
\begin{equation}  \int_{Z_{t + t_0}} \rho = \int_{Z_{t + t_0} \cap K} \rho = 
\int_{\vphi_t(Z_{t_0})} \rho. \label{eq_1146} \end{equation}
 We may now apply Lemma \ref{lem_1001_} with the vector field $V$ and with $Z = Z_{t_0}$. From (\ref{eq_1146}) and the conclusion of the lemma,
 $$ \left. \frac{d}{dt} \int_{Z_{t_0 + t}} \rho \right|_{t=0} = 
 \left. \frac{d}{dt} \int_{\vphi_t(Z)} \rho \right|_{t=0} = 
  -\int_{Z \cap U}  \langle V, N \rangle \cdot H_{\rho} =  -\int_{Z_{t_0} \cap U}  \frac{H_{\rho}}{|\nabla f|}, $$
  completing the proof.
\end{proof}

It is known that domains that are bounded by a smooth hypersurface of {\it constant weighted mean curvature} are critical points of the functional
$$ \Omega \mapsto \left( \int_{\partial \Omega} \rho \right) - \alpha \cdot \sigma(\Omega), $$
where $-\alpha$ is the weighted mean curvature, and where $\sigma$ is the Riemannian volume measure. Let us  demonstrate this criticality. Suppose that  $\Omega_0 \subseteq M$ is an open set with a smooth $(n-1)$-dimensional boundary $Z = \partial \Omega_0$, such that $H_{\rho}$ is constant on the hypersurface $Z$. Let us denote this constant by $-\alpha$.
According to Lemma \ref{lem_1001_}, for any vector field $V$ generating the flow $(\vphi_t)_{t \in \RR}$, setting $\Omega_t = \vphi_t(\Omega_0)$, 
$$ \left. \frac{d}{dt} \left( \int_{\partial \Omega_t} \rho - \alpha \cdot \sigma(\Omega_t) \right)\right|_{t=0}
=  \int_{Z} \langle V, N \rangle \cdot \alpha
\, - \, \alpha \cdot \int_Z \langle V, N \rangle = 0, $$
where the derivative of $\sigma(\Omega_t)$ is computed using the usual formula for the variation of Riemannian volume under deformations.

\begin{proposition} Let $M$ be a Riemannian manifold and let $f: M \rightarrow \RR$ be a smooth function.
Assume that $x \in M$ is a regular point of the function $f$.
Denote $t = f(x), Z_t = \{ f = t \}, \rho = |\nabla f|$ and $N = \nabla f / |\nabla f|$. Then at the point $x$,
$$ H_{\rho} = -\Delta f $$
where $H_{\rho}(x)$ is the weighted mean curvature of $Z_t$ at the point $x$ with respect to the normal $N$.
\label{prop_1444}
\end{proposition}

\begin{proof} Abbreviate $Z = Z_t$. We define, in a neighborhood of the point $x$, the signed distance function
	$$ r(y) = \left \{ \begin{array}{cc} d(y, Z) & f(y) \geq t \\ -d(y,Z) & f(y) < t \end{array} \right. $$
where $d(y, Z) = \inf_{z \in Z} d(y,z)$
and $d$ is the Riemannian distance function.
Then $r$ is a smooth function defined near $x$ such that $\nabla r = N = \nabla f/|\nabla f|$ at points of $Z$. Moreover, $\nabla_{\nabla r} \nabla r = 0$ since the integral curves of $\nabla r$ are geodesics.
Consequently, the usual mean curvature of the hypersurface $Z$ at the point $x$ with respect to the normal $N$ satisfies
\begin{equation}  -H = \Tr[ T_x Z \ni v \mapsto \nabla_v \nabla r ] = \Tr[ (\nabla^2 r)^{\sharp}|_{T_x Z} ] = \Tr[ (\nabla^2 r)^{\sharp} ] = \Delta r. \label{eq_1152}
\end{equation}
The functions $f - t$ and $r$ both vanish on the hypersurface $Z$ near the point $x$, with $\nabla r = \nabla f / |\nabla f| \neq 0$ at the point $x$.  It follows that
the function $u =  (f-t) / r$ is well-defined and smooth in a neighborhood of the point $x$, where we set $u(y) = |\nabla f(y)|$ for $y \in Z$.
Indeed, if $\gamma: (-\eps, \eps) \rightarrow M$ is a unit-speed geodesic with $\gamma(0) = x$ and $\dot{\gamma}(0) = N$, then by Taylor's theorem,
$$ f(\gamma(s)) -t = s |\nabla f(x)| + (s^2/2) \cdot \nabla^2 f(N, N) + o(s^2). $$
Consequently,
$$  \langle \nabla u, N \rangle = \lim_{s\rightarrow 0} \frac{u(\gamma(s)) - u(x)}{s} =
\lim_{s \rightarrow 0} \frac{(f(\gamma(s)) - t)/s - |\nabla f(x)|}{s} = \frac{1}{2} \nabla^2 f(N, N). $$
Therefore, since $\rho = \sqrt{|\nabla f|^2}$,
\begin{equation} \langle \nabla \rho, N \rangle =
\frac{N \langle \nabla f, \nabla f \rangle}{2 |\nabla f|} = \frac{\nabla^2 f(\nabla f, N)}{|\nabla f|}  = \nabla^2 f(N, N) = 2 \langle \nabla u, N \rangle. \label{eq_1212}
\end{equation}
According to (\ref{eq_1152}) and (\ref{eq_1212}), at the point $x \in Z$,
\begin{align*}
\Delta f & = \Delta (f-t) = \Delta (u r) = (\Delta u) r + 2 \langle \nabla u, \nabla r \rangle + u \Delta r  = 2 \langle \nabla u, N \rangle  + |\nabla f| \Delta r
\\ & = \langle \nabla \rho, N \rangle - \rho H = -  H_{\rho},
\end{align*}
as promised.
\end{proof}

Consider the case where $M$ is connected and where $f: M \rightarrow \RR$ is a non-constant Laplace eigenfunction, thus  $\Delta f = -\lambda f$ for some $\lambda > 0$. It follows from Proposition \ref{prop_1444}
that when $t$ is a regular value of $f$, the level set $Z_t = \{ f = t \}$ has a  constant weighted mean curvature $-\lambda t$, where the weight function is
\begin{equation}  \rho = |\nabla f| \label{eq_1101} \end{equation} and where the weighted mean curvature is computed with respect to the normal $\nabla f / |\nabla f|$.
In particular, if $0$ is a regular value of $f$,
then the nodal set $Z_0$ is a smooth hypersurface which is a {\it weighted minimal hypersurface}, i.e., its weighted mean curvature vanishes.
From Proposition \ref{lem_1001} and Proposition \ref{prop_1444} we also conclude the following:

\begin{corollary} For a compact Riemannian manifold $M$, a smooth function $f: M \rightarrow \RR$ and a regular value $t \in (\inf f, \sup f)$ of the function $f$, setting $Z_t = \{  f = t \}$,
	\begin{equation}  \frac{d}{dt} \int_{Z_t} |\nabla f| = \int_{Z_t}
	\frac{\Delta f}{|\nabla f|}. \label{eq_202} \end{equation}
	\label{cor_958}
\end{corollary}

\begin{remark} {\rm
Let  $f:M \rightarrow \RR$ be a smooth function defined on a compact Riemannian manifold $M$. For simplicity suppose that $f$ has only two critical points that are assumed to be non-degenerate (in this case  $M$ is diffeomorphic to a
sphere, by Morse theory). Say that we would like to find a  weight $\rho$ with respect to which all level sets of $f$ are weighted minimal hypersurfaces (except, of course, for the two critical points).
If such a weight $\rho$ is found, then we may consider the measure $\theta$ on $M$ whose density equals $\rho  |\nabla f|$. One may show (see the next section)
that the value distribution density of $f$ under $\theta$ is constant in the interval $(\inf f, \sup f)$.

\medskip However, it is not always possible to find such a smooth weight $\rho$.
The equation for $\log \rho$ that guarantees that all level sets of $f$ are weighted minimal hypersurfaces is
\begin{equation} \langle \nabla f, \nabla (\log \rho) \rangle = -\Delta f + \frac{\nabla^2 f(\nabla f, \nabla f)}{|\nabla f|^2}. \label{eq_546} \end{equation}
The partial differential equation (\ref{eq_546}) is solved by the method of characteristics, essentially by integrating the right-hand side of (\ref{eq_546}) along integral curves of $\nabla f$. It is not always the case that there exists a solution $\rho$ that is smooth in the entire manifold $M$. Indeed, any two integral curves of $\nabla f$ connect the two critical points of $f$,
yet the right-hand side of (\ref{eq_546}) could integrate to two different values along the two different integral curves.

\medskip We also think that in dimensions three and above, generically it is impossible to find a smooth weight $\rho$ with respect to which all level sets of $f$ have constant weighted mean curvature.
}
\end{remark}

\section{Eigenfunctions of the Laplacian}
\label{sec3}

The concepts of weighted mean curvature and weighted minimal surface are arguably helpful for forming an intuition regarding nodal sets. However, shorter proofs may sometimes be given without any appeal
to these concepts. For example, formula (\ref{eq_202}) may also be deduced  as a limit case of the following easy proposition:

\begin{proposition} For a compact, connected, Riemannian manifold $M$, a smooth function $f: M \rightarrow \RR$ and two regular values $t_1, t_2 \in (\inf f, \sup f)$ of the function $f$ with $t_1 < t_2$,
\begin{equation} 	\int_{Z_{t_2}} |\nabla f| - \int_{Z_{t_1}} |\nabla f| =  \int_{M_{t_1,t_2}} \Delta f, \label{eq_129} \end{equation}	
where $M_{t_1, t_2} = \{ t_1 < f < t_2 \}, Z_{t_i} = \{ f = t_i\}$ for $i=1,2$ and the integral over $M_{t_1,t_2}$ is carried out with respect to the Riemannian volume measure in $M$.
\label{prop_101}
\end{proposition}

\begin{proof} Consider the vector field $V = \nabla f$ whose divergence equals $\Delta f$, and apply the divergence theorem in the domain $M_{t_1, t_2}$.
\end{proof}

In the case where $f$ is an eigenfunction of the Laplacian, we
would like to eliminate the requirement that $t_1$ and $t_2$ are regular values from
Proposition \ref{prop_101}. The following lemma serves this purpose:

\begin{lemma} Let $M$ be a compact, connected, Riemannian manifold, let $\lambda > 0$ and let $f: M \rightarrow \RR$ be a non-constant, smooth function with $\Delta f = -\lambda f$. Then the function
	$$ \alpha(t) = \int_{Z_t} |\nabla f| \qquad \qquad \qquad (t \in \RR) $$
	is continuous in $\RR$, where the integral is carried out with respect to the $(n-1)$-dimensional Hausdorff measure.	 \label{lem_325}
\end{lemma}

\begin{proof} For any value $t \in \RR$, the set $Z_t = \{ f = t \}$ has a finite $(n-1)$-dimensional Hausdorff measure, and in fact,
\begin{equation}  \sup_{t \in \RR} Vol_{n-1}(Z_t) < \infty. \label{eq_217} \end{equation}
These known facts  follow from the compactness of $M$ and from a result by Aronszajn \cite{aron} according to which  $f - t$ cannot vanish to infinite order at a point of $Z_t$. Let us briefly sketch the argument, see also Donelly and Fefferman \cite{DF} and Hardt and Simon \cite{HS}. By compactness, it suffices to show that for any $p \in M$ there is a neighborhood $U$ with
\begin{equation}  \sup_{t \in \RR} Vol_{n-1}(U \cap Z_t) < \infty. \label{eq_1020} \end{equation}
We may thus pass to local coordinates $x^1,\ldots,x^n$ near $p$. In these local coordinates, we look at a non-zero Taylor polynomial $Q$ of degree $m$ that approximates $f$.
Consider the $m$-homogeneous part of $Q$, which is a non-zero polynomial in $\RR^n$. Pick
$n$ linearly-independent unit vectors in $\RR^n$ at which the $m$-homogeneous part of $Q$ does not vanish.

\medskip We have thus found $n$ linearly-independent directions in $\RR^n$, such that any line parallel to any of these directions intersects any level set of $Q$ at most $m$ times. The same holds for any level set of $f$ in these local coordinates: Otherwise, $f-t$ vanishes at $m+1$ points on such a line,
in contradiction to the proximity between the Taylor polynomial of $f$, restricted to this line, and the Lagrange interpolation polynomial with these $m+1$ nodes.
This provides a uniform bound for the volumes of the level sets of $f$, proving (\ref{eq_1020}).

\medskip For $\eps > 0$ we set
$$ \alpha_{\eps}(t) = \int_{Z_t} \eta(|\nabla f| / \eps) \cdot|\nabla f|   $$
where $\eta: \RR\rightarrow \RR$ is some fixed, smooth, non-decreasing function with $\eta(s) = 0$ for $s \leq 1/2$ and $\eta(s) = 1$ for $s \geq 1$. For $\eps > 0$, the
integrand $\rho_{\eps} =  \eta(|\nabla f| / \eps) \cdot |\nabla f| $ is smooth and compactly-supported in the open set $\{ |\nabla f| > 0 \}$. It thus follows from Proposition
\ref{lem_1001} that the function $\alpha_{\eps}$ is continuous in $\RR$, and in fact even differentiable. Denoting the supremum in (\ref{eq_217}) by $C > 0$,
we have $|\alpha_{\eps}(t) - \alpha(t)| \leq C \eps$ for all $t$. Thus $\alpha_{\eps}$ converges uniformly to $\alpha$, which is consequently a continuous function.
\end{proof}

\begin{corollary} The conclusion of Proposition \ref{prop_101}
	holds true for all $t_1, t_2 \in (\inf f, \sup f)$, provided that $f$ is non-constant and that $\Delta f = - \lambda f$  for some $\lambda > 0$. 	\label{cor_831}
\end{corollary}

\begin{proof} By Sard's lemma, equation
	(\ref{eq_129}) holds true for all $t_1 < t_2$ in a dense subset of the interval $(\inf f, \sup f)$. All that remains is to show that both the left-hand side and the right-hand side of (\ref{eq_129}) are continuous in $t_1$ and $t_2$.
	The continuity of the left-hand side is the content of Lemma \ref{lem_325}. The continuity of the right-hand side of (\ref{eq_129}) in $t_1$ and $t_2$ follows from the Lebesgue bounded convergence theorem.
\end{proof}

Next we will apply the coarea formula
for a smooth function $f: M \rightarrow \RR$, where $M$ is still a compact, connected, Riemannian manifold. This formula implies that for any  bounded, measurable, test function $\vphi: \RR \rightarrow \RR$,
\begin{equation}  \int_M \vphi(f) |\nabla f|^2 =
\int_{-\infty}^{\infty} \vphi(t)  \left(\int_{Z_t} |\nabla f| \right) dt, \label{eq_444} \end{equation}
with  $Z_t = \{ f = t \}$. Indeed,
formula (\ref{eq_444}) follows by
substituting $g = |\nabla f| \vphi(f)$
 in the statement of the coarea formula in Federer \cite[Theorem 3.1]{federer},
see also Evans and Gariepy \cite{EG}.

\begin{proof}[Proof of Theorem \ref{thm_1033}] Since $M$ is compact and connected and $f$ is non-constant,
 $\Delta f = -\lambda f$ for some $\lambda > 0$. Write $\mu$ for the measure on $M$ whose density is $|\nabla f|^2$.
 Formula (\ref{eq_444}) states that the push-forward measure $f_* \mu$ has a density in $\RR$ that is equal to $\psi(t) = \int_{Z_t} |\nabla f|$.

\medskip The function $\psi$ is the un-normalized value distribution density of $f$ under $\mu$, in the sense that $\frac{1}{\mu(M)} \cdot \psi$ is the actual value distribution density as defined above. The function $\psi$ is continuous in $\RR$, according to Lemma \ref{lem_325}.
	Let us prove that $\psi$ is decreasing on $[0, \sup f]$. For any $t_1, t_2 \in [0, \sup f]$ with $t_1 < t_2$, by Corollary
	\ref{cor_831},
	$$ \psi(t_2) - \psi(t_1) = \int_{Z_{t_2}} |\nabla f| -
	\int_{Z_{t_1}} |\nabla f| = \int_{M_{t_1,t_2}} \Delta f =
-\lambda \cdot \int_{M_{t_1,t_2}} f < 0, $$	
since $M_{t_1, t_2} = \{ t_1 < f < t_2 \}$ is a set of positive measure in which $f$ is positive. The proof that $\psi$ is increasing in $[\inf f, 0]$ is similar.
\end{proof}

\begin{remark} {\rm
As can be seen from the proof of Theorem \ref{thm_1033}, the un-normalized value distribution density $\psi$ of the function $f$ under $\mu$ is absolutely continuous in $\RR$. Its derivative equals the function $-\lambda t g(t)$, where $g$ is the un-normalized value distribution density of $f$ under the Riemannian volume measure.}
\end{remark}

\begin{remark} {\rm
Apart from regularity issues, in the proof of Theorem \ref{thm_1033}
we have not made any use of the eigenfunction equation $\Delta f = -\lambda f$ beyond the equality of signs.
$$ \sgn(\Delta f) = -\sgn(f). $$}
\end{remark}

\section{Manifolds with boundary}
\label{sec4}

Let us generalize Theorem
\ref{thm_1033} to the case
of the {\it weighted Laplacian}.
Let $(M, g, \rho)$ be an $n$-dimensional, weighted, Riemannian manifold, and assume that $\rho$ is positive everywhere. The usual definition of the weighted Laplacian $L$ of the weighted manifold $(M,g, \rho)$ is
\begin{equation} L u = \Delta u + \langle \nabla (\log \rho), \nabla u \rangle \label{eq_944} \end{equation}
for a smooth function $u: M \rightarrow \RR$. See Bakry, Gentil and Ledoux \cite{BGL} and
Grigor'yan \cite{gri}
for background on weighted Riemannian manifolds and their associated Laplacian. The weighted Laplacian $L$  satisfies the integration by parts formula
$$ \int_M (L u) v d \sigma = -\int_M \langle \nabla u, \nabla v \rangle d \sigma, $$
where in this section we write $\sigma$ for the
measure on $M$ whose density with respect to the Riemannian volume measure is $\rho$. We view $\sigma$ as the uniform measure on the weighted Riemannian manifold $(M, g, \rho)$.
Moreover, we have $L u = \div_{\sigma} (\nabla u)$, where for a vector field $V$ on $M$,
$$ \div_\sigma(V) = \div(V) + \langle \nabla (\log \rho), V \rangle $$
and $\div$ is the usual Riemannian divergence. The divergence theorem shows that for any domain $\Omega \subseteq M$ with a compact closure and a smooth $(n-1)$-dimensional boundary, and for any smooth vector field $V$ on $M$,
\begin{equation}  \int_{\partial \Omega} \langle V, N \rangle \rho = \int_{\Omega} \div_\sigma(V) d \sigma, \label{eq_1524} \end{equation}
where $N$ is the outer unit normal to $\partial \Omega$. Theorem \ref{thm_952}  follows from the case $\rho \equiv 1$ of the following:

\begin{theorem} Let $(M, g, \rho)$ be a connected, weighted Riemannian manifold, such that $\rho$ is positive everywhere. Write $\sigma$ for the measure whose density with respect to the Riemannian volume measure on $M$ equals $\rho$. Let $f: M \rightarrow \RR$ be a non-constant eigenfunction of the weighted Laplacian $L$.
	Consider the measure $\mu$ whose density with respect to $\sigma$ is the function $|\nabla f|^2$. Assume one of the following:
	\begin{enumerate}
		\item[(i)] The case of a closed manifold: The manifold $M$ is a compact  manifold without boundary.
		\item[(ii)] The case of Dirichlet boundary conditions:  The manifold $M$ has a smooth, $(n-1)$-dimensional boundary $\partial M$ and $f|_{\partial M} \equiv 0$. The space $M \cup \partial M$ is compact,
and $f$ is smooth up to the boundary.
		\item[(iii)] The case of Neumann boundary conditions: The manifold $M$ has a smooth $(n-1)$-dimensional boundary $\partial M$ with outer unit normal $N$, and $\langle \nabla f, N \rangle = 0$ on the boundary. The space $M \cup \partial M$ is compact, and $f$ is smooth up to the boundary.
	\end{enumerate}
	
	Then in each of these three cases,
	 the value distribution density of the function $f$ under $\mu$ is strictly-increasing in $[\inf f, 0]$ and strictly-decreasing in $[0, \sup f]$.  \label{thm_540}
\end{theorem}

\begin{proof} In each of these three cases, it follows from the divergence theorem with $V = \nabla f$ that for any two regular values $t_1, t_2 \in (\inf f, \sup f)$ with $t_1 < t_2$ and $\sgn(t_1) = \sgn(t_2)$,
	\begin{equation}
	\int_{Z_{t_2}} |\nabla f| \rho -
	\int_{Z_{t_1}} |\nabla f| \rho = \int_{M_{t_1, t_2}} L f d \sigma, \label{eq_534}
	\end{equation}
	where $Z_t = \{ f = t\}$ and $M_{t_1,t_2} = \{ t_1 < f < t_2 \}$.
Indeed, in case (i) the boundary of $M_{t_1, t_2}$ is the smooth manifold $Z_{t_1} \cup Z_{t_2}$, and (\ref{eq_534}) holds true in view of (\ref{eq_1524}). The same holds in case (ii), since $\sgn(t_1) = \sgn(t_2)$. In case (iii), the boundary of $M_{t_1, t_2}$ includes $Z_{t_1}, Z_{t_2}$ and also some parts of $\partial M$. It is still legitimate to apply the divergence theorem in a domain with a piecewise smooth boundary, see Evans and Gariepy \cite{EG}. Since $\langle V, N \rangle = \langle \nabla f, N \rangle = 0$ on the boundary $\partial M$, there is no flux going through $\partial M$ and so (\ref{eq_534}) holds true in case (iii) as well.
	
\medskip 	
The proof that $t \mapsto \int_{Z_t} |\nabla f| \rho$ is continuous is completely  analogous to the argument of Lemma \ref{lem_325} in case (i) and  in case  (iii), as well as in case (ii) for non-zero values of $t$. Similarly to the proof of Corollary \ref{cor_831}, we conclude 	that (\ref{eq_534}) holds true without the requirement that $t_1$ and $t_2$ be regular values.
	Lastly, the use of the coarea formula in the proof of Theorem \ref{thm_1033} as well as the rest of the argument are adapted in a straightforward manner to the weighted case.
\end{proof}


\section{Monotonicity formula for  solid spherical harmonics}
\label{sec5}

In this section we no longer discuss general Riemannian manifolds, but we specialize to the case of $\RR^n$ with its Euclidean metric.  Let us fix a smooth hypersurface $M \subseteq \RR^n$ with a compact closure $\overline{M}$ and a smooth, $(n-2)$-dimensional boundary $\partial M = \overline{M} \setminus M$.
We also fix a smooth weight function $\rho$ defined in a neighborhood of the closure of $M$. Assume that $M$ is a weighted minimal hypersurface, i.e.,
\begin{equation}  H_{\rho} = 0 \label{eq_1134} \end{equation}
everywhere in $M$.
In this section we use a normalization of the $\rho$-divergence that is slightly different from the  one we had in Section \ref{sec4}.
For a vector field
$V$ defined in a neighborhood of $\overline{M}$, we set
\begin{equation}  \div_{M,\rho}(V) \, = \, \sum_{i=1}^{n-1} \rho \langle \partial_{e_i} V, e_i \rangle \, + \,
\langle \nabla \rho, V \rangle
 \label{eq_124} \end{equation}
where $e_1,\ldots,e_{n-1}$ is an orthonormal basis of $T_p M$, and the definition clearly does not depend on the choice of this orthonormal basis.
We write $\pi_M(V)$ for the orthogonal projection of the  vector $V(x) \in T_x \RR^n$ to the hyperplane $T_x M \subset T_x \RR^n$ at a point $x \in M$.
Writing $$ \div_M(V) = \sum_{i=1}^{n-1}  \langle \partial_{e_i} V, e_i \rangle $$ we have
\begin{equation}  \div_M( \rho \pi_M( V)) = \rho \div_M(\pi_M(V)) + \langle \nabla \rho,  \pi_M(V) \rangle = \div_{M, \rho}(\pi_M(V)). \label{eq_114} \end{equation}
The fact that $M$ is a weighted minimal hypersurface implies that
\begin{equation}
\div_{M,\rho}(  V) = \div_{M,\rho}( \pi_M(V)). \label{eq_1147}
\end{equation}
Indeed, by linearity it suffices to prove (\ref{eq_1147}) for a vector field $V$ that satisfies $\pi_M(V) = 0$.
Thus, locally near a point $x \in M$, we may assume that the restriction of $V$ to $M$ takes the form $V = f N$, where $f$ is a scalar function  and $N$ is a unit normal to $M$,  defined and smooth in a neighborhood of $x$.
 In this case, at $x \in M$,
$$ \div_{M,\rho}(V) =
\sum_{i=1}^{n-1} \rho \langle \partial_{e_i} (f N), e_i \rangle + f \langle \nabla \rho, N \rangle = f (-\rho H + \langle \nabla \rho, N \rangle ) = -f H_{\rho} = 0, $$
where the usual mean curvature $H$  and the weighted mean curvature $H_{\rho}$ are computed with respect to the normal $N$. The important equality (\ref{eq_1147}) is thus proven. From (\ref{eq_114}), (\ref{eq_1147}) and the divergence theorem on $M$ we conclude that
\begin{equation}  \int_{\partial M} \langle V, \nu \rangle \rho =
\int_{\partial M} \langle \rho \pi_M(V), \nu \rangle  =
\int_M \div_M( \rho \pi_M( V))
= \int_M \div_{M, \rho}( V) \label{eq_1231}
\end{equation}
for any smooth vector field $V$ in $\RR^n$, where $\nu$ is the outer unit normal to $\partial M$ relative to $M$, and where the
integrals are carried out with respect to the Hausdorff measures of the corresponding dimensions.

\begin{proposition}
	Let $P: \RR^n \rightarrow \RR$ be a a harmonic, $k$-homogeneous polynomial with $k \geq 1$. Fix a regular value $t_0 \neq 0$ of the polynomial $P$ and set
$M = \{ P = t_0 \}$. For $r > 0$ denote $M_r = M \cap B(0, r)$.  Then for any $r > 0$ for which $M_r \neq \emptyset$,
$$  \frac{n+k-2}{r} \int_{M_r} |\nabla P| = \frac{d}{dr} \left[ \int_{M_r} |\nabla P| - k^2 t_0^2 \int_{M_r}
\frac{1}{ |x|^2 |\nabla P|}  \right], $$ \label{prop_1134}
where all integrals are carried out with respect to the $(n-1)$-dimensional Hausdorff measure.
\end{proposition}

\begin{proof} 	Consider the weight function $\rho = |\nabla P|$, which is positively-homogeneous of degree $k-1$ in $\RR^n$ and is smooth in a neighborhood of $M$.
Then for the vector field $V(x) = x$,
	$$ \div_{M, \rho}(V) = \rho (n-1) + \langle \nabla \rho, x \rangle  = (n+k-2) \rho, $$
	since $\rho$ is positively-homogensous of degree $k-1$.
	The set $M$ is a closed, smooth hypersurface in $\RR^n$.
	 The set $M_r$ is a smooth hypersurface whose boundary $\partial M_r$ relative to $M_r$ is contained in $\partial B(0,r)$. Since $P$ is harmonic,
	Proposition \ref{prop_1444} implies that $H_{\rho} = 0$ on $M_r$, and hence $M_r$ is a weighted minimal hypersurface. We will use the divergence theorem in the form of formula (\ref{eq_1231}) for the vector field $V(x) = x$, and obtain
\begin{equation} (n+k-2)  \int_{M_r} \rho
= \int_{M_r} \div_{M, \rho}( V) = \int_{\partial M_r} \langle x, \nu \rangle \rho. \label{eq_1036}
\end{equation}
Write $f(x) = |x|$ for $x \in \RR^n$, so $\nabla f = x / |x|$ for $x \neq 0$. Recall that $\pi_M: T_x \RR^n \rightarrow T_x M$ is the orthogonal projection operator for $x \in M$.
Since $M_r = M \cap \{ f < r \}$, the outer unit normal to $\partial M_r$ relative to $M_r$ is
$$ \nu = \frac{\pi_M \nabla f}{|\pi_M \nabla f|}, $$
at any $x \in \partial M_r$ which is a regular point of the function $f|_M: M \rightarrow \RR$.
From (\ref{eq_1036}), for any  regular value $r > 0$ of the function $f|_M$,
\begin{equation}
(n+k-2) \int_{M_r} \rho = \int_{\partial M_r} \langle x,\nu \rangle \rho = \int_{\partial M_r}  \left \langle |x| \nabla f, \frac{ \pi_M \nabla f }{|\pi_M \nabla f|} \right \rangle \rho =
r \int_{\partial M_r} |\pi_M \nabla f| \rho. \label{eq_654}
\end{equation}
Next we use the coarea formula,  whose particular case appears as equation (\ref{eq_444}) above. Note that $\pi_M \nabla f$ is the Riemannian gradient of the function $f$ on the manifold $M$. By substituting
$g = \rho |\pi_M \nabla f|$ in \cite[Theorem 3.1]{federer}, we see that for any $R > 0$,
\begin{equation}  \int_{M_R} |\pi_M \nabla f|^2 \rho = \int_0^R \left(  \int_{\partial M_r} |\pi_M \nabla f| \rho \right) dr = \int_0^R \left( \frac{n+k-2}{r} \int_{M_r}\rho \right)dr, \label{eq_1054} \end{equation}
where we used (\ref{eq_654}) in the last passage, in accordance with Sard's lemma which states that almost any value $r \in (0, R)$ is a  regular value of the function $f|_M$.  We claim that the function $r \mapsto \int_{M_r} \rho$ is a continuous function of $r > 0$. By the Lebesgue bounded convergence theorem, this would follow once we know that the set $\partial M_r$ has a zero $(n-1)$-dimensional measure. However, if $\partial M_r$ had a positive $(n-1)$-dimensional measure, then the restriction of $P$ to $S^{n-1}$ would be a spherical harmonic with a level set of positive volume, in contradiction to Lemma \ref{lem_325}.

\medskip Thus the integrand on the right-hand side of (\ref{eq_1054}) is a continuous, non-negative function of $r$. We may therefore differentiate (\ref{eq_1054}), and obtain that for all $r >0$,
\begin{equation} \frac{d}{dr} \int_{M_r}
|\pi_M \nabla f|^2 \rho =  \frac{n+k-2}{r} \int_{M_r} \rho. \label{eq_1101_} \end{equation}
The vector field $\nabla P / |\nabla P|$ is a unit normal to $M$  while $\nabla f = x / |x|$. Therefore, at $x \in M$,
\begin{equation}  |\pi_M \nabla f|^2 = |\nabla f|^2 - \left \langle \nabla f, \frac{\nabla P}{|\nabla P|} \right \rangle^2 = 1 - \frac{k^2 P^2}{|x|^2 |\nabla P|^2}
= 1 - \frac{k^2 t_0^2}{|x|^2 |\nabla P|^2}. \label{eq_658} \end{equation}
From (\ref{eq_1101_}) and (\ref{eq_658}),
$$ \frac{n+k-2}{r} \int_{M_r} \rho = \frac{d}{dr} \left[ \int_{M_r} \rho - k^2 t_0^2 \int_{M_r} \frac{\rho}{|x|^2 |\nabla P|^2} \right], $$
		completing the proof.
\end{proof}

Our proof of Theorem \ref{thm_602} is similar to the proof of the monotonicity formula in minimal surface theory (see e.g. Simon \cite[Chapter 4]{simon}).

\begin{proof}[Proof of Theorem \ref{thm_602}] Consider first the case where $t$ is a regular, non-zero value of $P$. Denote
	$I(r) = \int_{M_r} |\nabla P|$
	and $e = n+k-2$. Let $r_0 = \inf_r M_r \cap B(0,r) \neq \emptyset$. According to Proposition \ref{prop_1134}, for any $r > r_0$,
\begin{equation}  \frac{d}{dr} \left[ \frac{I(r)}{r^{e}} \right] = \frac{k^2 t_0^2}{r^{e}} \cdot \frac{d}{dr} \int_{M_r} \frac{1}{|x|^2 |\nabla P| } > 0. \label{eq_1145} \end{equation}
Thus the function $I(r)/ r^e$ vanishes in $(0, r_0)$, and it is increasing in $[r_0, \infty)$, as advertised.

\medskip It still remains to deal with the case where $t$ is a non-zero, non-regular value of $P$ (the case where $t = 0$ is trivial, as remarked above). An argument analogous to the proof of Lemma \ref{lem_325} shows that for any fixed $r > 0$, the function
$$ t \mapsto \int_{\{ P = t \} \cap B(0,r)} |\nabla P| $$
is continuous in $t$. This implies the closedness of the set of all $t \in \RR$ for which the monotonicity  conclusion of the theorem holds true. According to Sard's lemma, the collection of regular, non-zero values is a dense set. Hence the conclusion of the theorem holds true for all values of $t$.	
\end{proof}

It is possible to rephrase Theorem \ref{thm_602} in terms of a spherical harmonic $p: S^{n-1} \rightarrow \RR$ of degree $k \geq 1$, rather than in terms of its $k$-homogeneous extension $P: \RR^n \rightarrow \RR$, which is a solid spherical harmonic. Fix $t > 0$ and set $Z_t = \{ P = t \} \subseteq \RR^n$. By applying the change-of-variables $y = r x / |x|$ we obtain
$$		\int_{Z_t \cap B(0,r)} |\nabla P| = \int_{\{ P \geq t \} \cap \partial B(0, r)} \left( \frac{t}{P} \right)^{(n-1)/k} \cdot \frac{1}{\left| \left \langle \frac{x}{|x|}, \frac{\nabla P}{|\nabla P|} \right \rangle \right| } \cdot \left( \frac{t}{P} \right)^{(k-1)/k} \cdot |\nabla P|.  $$
Consequently,
\begin{equation}  r^{-(n+k-2)} \int_{Z_t \cap B(0,r)} |\nabla P| =
		\frac{(t/r^k)^{(n+k-2)/k}}{k}  	\cdot
		\int_{\{ p \geq t / r^{k}   \} \cap S^{n-1}} \frac{|\nabla_S p|^2 + k^2 p^2}{p^{(n+2k-2)/k}}, \label{eq_425}  \end{equation}
		where $\nabla_S p$ is the spherical gradient of $p$, and where we used the fact that $|\nabla P|^2 = |\nabla_S p|^2 + k^2 p^2$ on the sphere $S^{n-1}$.
		We write $\Delta_S$ for the Riemannian Laplacian on $S^{n-1}$.
		Theorem \ref{thm_602} thus implies the following:
		
		\begin{corollary} Let $p: S^{n-1} \rightarrow \RR$ be a spherical harmonic of degree $k \geq 1$, i.e., $\Delta_{S} p = -k(n+k-2) p$. Then
			$$ \eps \mapsto \eps^{(n+k-2)/k} \int_{\{ p \geq \eps   \} \cap S^{n-1}} \frac{|\nabla_S p|^2 + k^2 p^2}{p^{(n+2k-2)/k}} $$
			is a non-increasing function of $\eps \in (0, \infty)$, where the integral is carried out with respect to the uniform measure on $S^{n-1}$, the $n$-dimensional Hausdorff measure. \label{cor_929}
		\end{corollary}
	
\begin{proof}[Proof of Corollary \ref{cor_834}] From Theorem \ref{thm_602} we know that the left-hand side of (\ref{eq_425}) is non-decreasing in $r > 0$ for any fixed positive $t$.
However, the right-hand side of (\ref{eq_425}) is a function of $t / r^k$, and consequently it is non-increasing in $t > 0$ for any fixed positive $r$. Setting $r = 1$, we see that the function
\begin{equation}  \psi(t) =  \int_{Z_t \cap B(0,1)} |\nabla P| \qquad \qquad (t \in \RR) \label{eq_850}
\end{equation}
is non-increasing in $[0, \infty)$. The proof that the function $\psi$ is non-decreasing in $(-\infty, 0]$ is completely analogous.
Thus the function in (\ref{eq_852}) is unimodal with a maximum at zero.
As in the proof of Theorem \ref{thm_1033}, from the coarea formula we conclude that the value distribution density of $P$ under $\mu$ is a constant multiple of
$\psi$. The proof is complete.
\end{proof}

\begin{remark}{\rm  It was suggested to us by Emanuel Milman to look at the case of Robin boundary conditions, since these are satisfied by solid spherical harmonics. This indeed leads to an alternative, much simpler  proof of Corollary \ref{cor_834}. Recall that a function $f$ on a manifold $M$ with boundary $\partial M$ is said to satisfy the Robin boundary conditions if
		$$ a f + b \langle \nabla f, N \rangle = 0 \qquad \qquad \textrm{on} \ \partial M. $$
for some parameters $a,b \in \RR$ with $a^2 + b^2 > 0$. A solid spherical harmonic $f: \RR^n \rightarrow \RR$ of degree $k$ on the unit ball $B(0,1)$ satisfies the Robin boundary conditions with $a = k, b = -1$. Arguing as in the proof of Theorem \ref{thm_1033}, for $t_2 > t_1 > 0$,
$$ \int_{Z_{t_2}} |\nabla f| - \int_{Z_{t_1}} |\nabla f| = -\int_{M_{t_1, t_2} \cap \partial B(0,1)} k f + \int_{M_{t_1, t_2} \cap B(0,1)} \Delta f = -\int_{M_{t_1, t_2} \cap \partial B(0,1)} k f < 0,	$$
with $M_{t_1, t_2} = \{ x \in \RR^n \, ; \, t_1 < f(x) < t_2 \}$ and $Z_{t_i} = \{ x \in B(0,1) \, ; \, f(x) = t_i \}$. This immediately implies Corollary \ref{cor_834}, which in turn is equivalent to Theorem \ref{thm_602}.	
}\end{remark}

\bigskip
\noindent Department of Mathematics, Weizmann Institute of Science, Rehovot 76100, Israel. \\
{\it e-mail:} \verb"boaz.klartag@weizmann.ac.il"

\end{document}